\documentclass[a4paper,oneside,reqno,12pt]{amsart}                     
\usepackage{graphicx}
\usepackage{graphics}

\usepackage{amssymb, amsmath, amsfonts}
\newtheorem{thm}{Theorem}[section]
\newtheorem{lem}{Lemma}[section]

\newtheorem{defn}{Definition}[section]
\newtheorem{ex}{Example}[section]

\newtheorem{rmk}{Remark}[section]
\newcommand{\Real}{{\mathbb R}}

\newcommand{\To}{\longrightarrow}
\newcommand{\toto}{\rightrightarrows}
\def\1{\^{\i}}
\def\2{\u{a}}
\def\3{\c{s}}
\def\4{\^{a}}
\def\5{\c{t}}
\def\a{\alpha}

\def\e{\epsilon}
\def\d{\delta}

\def\l{\lambda}
\def\<{\langle}
\def\>{\rangle}
\DeclareMathOperator*\dom{dom}

\DeclareMathOperator*\cl{cl}
\DeclareMathOperator*\co{co}

\DeclareMathOperator*\core{core}

\DeclareMathOperator*\inte{int}
\usepackage{mathptmx}      
\begin{document}

\title{Densely defined equilibrium problems}



\author{Szil\'ard L\' aszl\' o  \and Adrian Viorel}

\address{Szil\'ard L\' aszl\' o, 
             Department of Mathematics, Technical University of Cluj-Napoca,
              Str. Memorandumului nr. 28, 400114 Cluj-Napoca, Romania.}
              \email{laszlosziszi@yahoo.com}           
\address{Adrian Viorel,
             Department of Mathematics, Technical University of Cluj-Napoca,
              Str. Memorandumului nr. 28, 400114 Cluj-Napoca, Romania.}
              \email{adrian.viorel@math.utcluj.ro}

\maketitle

\begin{abstract}
In the present work we deal with set-valued equilibrium problems for which we provide sufficient conditions for the existence of a solution. The conditions that we consider are imposed not on the whole domain, but rather on a self segment-dense subset of it, a special type of dense subset. As an application, we obtain a generalized Debreu-Gale-Nika\"{\i}do-type theorem, with a considerably weakened Walras law in its hypothesis. Further, we consider a non-cooperative $n$-person game and prove the existence of a Nash equilibrium, under assumptions that are less restrictive than the classical ones.
\end{abstract}

\section{Introduction}

Equilibrium problems play an important role in nonlinear analysis especially due to
their implications in mathematical economics, and, due to its key  applications, Ky Fan's minimax
inequality is considered to be the most notable existence result in this field (see \cite{Fan1}. 

More recently, A. Krist\'aly and Cs. Varga \cite{KV} were able to prove two
set-valued versions of Ky Fan's inequality. Their results guarantee
the existence of solutions to set-valued equilibrium problems that
they have introduced, motivated by Browder's study of variational inclusions
\cite{Browder}.


The present work is devoted to set-valued equilibrium problems as
defined in \cite{KV}. More precisely, we show that the hypotheses of the
existence result of Krist\'aly and Varga (Theorems 2.1 and 2.2 in \cite{KV})
can be weakened in the sense that the convexity and continuity assumptions must not hold on
the whole domain, but just on a special type of dense subset of it that we
call \emph{self segment-dense}(see \cite{LaVi}).%

This new concept is related to, but different from that of a segment-dense
set introduced by Dinh The Luc \cite{DTL} in the context of densely
quasimonotone, respectively densely pseudomonotone operators. In one
dimension, the concepts of a segment-dense set respectively a self
segment-dense set are equivalent to the concept of a dense set. Nevertheless,
in dimension greater than one, self segment-dense subsets enjoy certain
special properties, characterized by Lemma \ref{l51}, which play a crucial
role in obtaining our existence results.

We explore the role of self segment-dense sets in the context of equilibrium problems both %
with and without compactness assumptions, and show how our abstract results can be applied.

The two applications that we have in mind concern the theory of economic equilibrium and game %
theory. In fact, we prove a result of Debreu-Gale-Nika\"{\i}do type (see \cite{Debreu},\cite{Gale},\cite{Nikaido}), that states the existence %
of an economic equilibrium even if the constraint imposed by Walras' law holds only on a %
self segment-dense subset of the price simplex. Our second result proves the existence of Nash %
equilibria for non-cooperative $n$-person games under assumptions that are more %
general than those of the classical theory (cf. \cite{Aubin}).

In an infinite dimensional real Hilbert space it is known that the unit sphere %
is dense in the unit ball with respect to the weak topology, but, as we will see, it is not %
self segment-dense. This is a typical example of a dense set that is not self segment-dense. %
Using this example, in Section 3, we argue that it is not enough to impose %
the convexity and continuity assumptions on a dense subset of the domain, and that %
it is essential to work with a self segment-dense subset.

The outline of the paper is the following. In the next section we formulate
the problems that we are dealing with and introduce the necessary apparatus.
We also define the notion of a self segment-dense set and show by an example
that it differs from the notion of a segment-dense set introduced in \cite%
{DTL}. Sections 3 and 4 contain the main result of our work, namely existence
results for both set-valued an single-valued equilibrium problems. The final two %
sections contain applications of our abstract results. In Section 5 %
we prove a generalized Debreu-Gale-Nika\"{\i}do-type theorem while in Section 6 %
we obtain the existence of a non-cooperative equilibrium. The corresponding %
results in \cite{Aubin} are generalized.


\section{Preliminaries}


In what follows $X$  and $Y$ denote  Hausdorff topological spaces.
For a non-empty set $D\subseteq X$, we denote by $\inte(D)$ its interior and
by $\cl(D)$ its closure. We say that $P\subseteq D$ is dense in $D$ if $%
D\subseteq \cl(P)$, and that $P\subseteq X$ is closed regarding $D$ if $\cl%
(P)\cap D=P\cap D.$

Let $T:X\rightrightarrows Y$ be a set-valued operator. We denote by $D(T)=\{x\in X:T(x)\neq \emptyset \}$ its domain and by $R(T)=\displaystyle\bigcup_{x\in D(T)}{T(x)}$ its range. The graph of the operator $T$ is the set $G(T)=\{(x,y)\in X\times Y:\,y\in T(x)\}.$

Recall that $T$ is said to be \emph{upper semicontinuous} at $x\in D(T)$ if
for every open set $N\subseteq Y$ containing $T(x)$, there exists a
neighborhood $M\subseteq X$ of $x$ such that $T(M)\subseteq N.$ $T$ is said
to be \emph{lower semicontinuous} at $x\in D(T)$ if for every open set $%
N\subseteq Y$ satisfying $T(x)\cap N\neq \emptyset $, there exists a
neighborhood $M\subseteq X$ of $x$ such that for every $y\in M\cap D(T)$ one
has $T(y)\cap N\neq \emptyset .$ $T$ is upper semicontinuous (lower
semicontinuous) on $D(T)$ if it is upper semicontinuous (lower
semicontinuous) at every $x\in D(T).$

With $T$ as before and $V\subseteq Y$, let us introduce the following sets
$$T^{-}(V)=\{x\in X:T(x)\cap V\neq\emptyset\},$$
and
$$T^{+}(V)=\{x\in X:T(x)\subseteq V\},$$
called the inverse image of $V$, respectively the core of $V.$

\begin{rmk}\rm\label{r21} Let $T:X\toto Y$ be a set valued map. The following characterizations of lower semicontinuity, respectively upper semicontinuity (see \cite{AF}) can easily be proved.

\begin{itemize}
\item[(i)]$T$ is lower semicontinuous at $x\in D(T)$ if and only
if for every net $(x_{\a})\subseteq D(T)$ such that $x_\a\longrightarrow
x$ and for every $x^*\in T(x)$ there exists a net $x^*_{\a}
\in T(x_\a)$ such that $x^*_{\a}\longrightarrow x^*.$
\item[(ii)] $T$ is upper semicontinuous at $x\in D(T)$ if and only
if for every net $(x_{\a})\subseteq D(T)$ such that $x_\a\longrightarrow
x$ ad every open set $V\subseteq Y$ such that $T(x)\subseteq V$ one has $F(x_\a)\subseteq V$ for sufficiently large $\a.$
\item[(iii)] $T$ is lower semicontinuous if and only if for all closed set $V\subseteq Y$,  one has $T^{+}(V)$ is closed in $X.$
\item[(iv)]  $T$ is upper semicontinuous if and only if for all closed set $V\subseteq Y$,  one has $T^{-}(V)$ is closed in $X.$
\end{itemize}
\end{rmk}

Obviously, when $T$ is single-valued, then upper semicontinuity and also
lower semicontinuity become the usual notion of continuity.

For a  function $f:X\longrightarrow\overline{{\mathbb{R}}}={\mathbb{R}}\cup\{+\infty\}$ we denote by $\dom f$ its domain, that is $\dom f=\{x\in X:f(x)\in \Real\}.$

We say that $f$ is upper semicontinuous at $x_0\in \dom f$ if for every $\e>0$ there exists a neighborhood  $U$ of $x_0$ such that $f(x)\le f(x_0)+\e$ for all $x\in U.$  The function $f$ is called upper semicontinuous if it is upper semicontinuous at every point of its domain.

Also, we say that $f$ is lower semicontinuous at $x_0\in dom f$ if for every $\e>0$ there exists a neighborhood  $U$ of $x_0$ such that $f(x)\ge f(x_0)-\e$ for all $x\in U.$
The function $f$ is called lower semicontinuous if it is lower semicontinuous at every point of its domain.
\begin{rmk}\rm\label{r22} Let $f:X\longrightarrow\overline{\mathbb{R}}$ be a function. Then, we have the following characterizations of the lower semicontinuity, respectively the upper semicontinuity of $f$:

\begin{itemize}
\item[(i)] $f$ is upper semicontinuous at $x_0,$ if and only if, $\limsup_{x^\a\to x_{0}} f(x^\a)\le f(x_0),$ where $(x^\a)$ is a net converging to $x_0.$

\item[(ii)] $f$ is lower semicontinuous at $x_0,$ if and only if, $\liminf_{x^\a\to x_{0}} f(x^\a)\ge f(x_0),$ where $(x^\a)$ is a net converging to $x_0.$

\item[(iii)] $f$ is upper semicontinuous on $X$, if and only if, the superlevel set $\{x\in X: f(x)\ge a\}$ is a closed set for every $a\in \Real.$

\item[(iv)] $f$ is lower semicontinuous on $X$, if and only if, the sublevel set $\{x\in X: f(x)\le a\}$ is a closed set for every $a\in \Real.$
\end{itemize}
\end{rmk}


\subsection{Set-valued equilibrium problems}


Let $X$ be a real normed space, let $K\subseteq X$ be a nonempty set and let $F:K\times K\rightrightarrows {%
\mathbb{R}}$ be a set valued map. According to \cite{KV} a set-valued
equilibrium problem consists in finding $x_0\in K$ such that
\begin{equation*}
F(x_0,y)\geq 0,\,\forall y\in K.
\end{equation*}%
Here $F(x,y)\geq 0$ means that $u\geq 0$ for all $u\in F(x,y),$ or, in other
words, that $F(x,y)\subseteq \lbrack 0,\infty )={\mathbb{R}}_{+}.$

A different set-valued equilibrium problem, also formulated in \cite{KV}, is
to find $x_0\in K$ such that
\begin{equation*}
F(x_0,y)\cap {\mathbb{R}}_-\neq\emptyset,\,\forall y\in K.
\end{equation*}

For the the convenience of the reader, we recall the original existence results %
of Krist\'aly and Varga regarding the two set-valued equilibrium problems.

\begin{thm}
\label{tKV1} Let $X$ be a real normed space, $K$ be a nonempty convex
compact subset of $X,$ and $F :K \times K \rightrightarrows {\mathbb{R}}$ a
set valued map satisfying

\begin{itemize}
\item[(i)] $\forall y\in K, x\longrightarrow F(x, y)$ is lower
semicontinuous on $K$,

\item[(ii)] $\forall x\in K, y\longrightarrow F(x, y)$ is convex on $K$,

\item[(iii)] $\forall x\in K, F(x,x)\ge 0.$
\end{itemize}

Then, there exists an element $x_0 \in K$ such that
\begin{equation*}
F(x_0, y)\ge 0, \,\forall y\in K.
\end{equation*}
\end{thm}

\begin{thm}
\label{tKV2} Let $X$ be a real normed space, $K$ be a nonempty convex
compact subset of $X,$ and $F :K \times K \rightrightarrows {\mathbb{R}}$ a
set valued map satisfying

\begin{itemize}
\item[(i)] $\forall y\in K, x\longrightarrow F(x, y)$ is upper
semicontinuous on $K$,

\item[(ii)] $\forall x\in K, y\longrightarrow F(x, y)$ is concave on $K$,

\item[(iii)] $\forall x\in K, F(x,x)\cap {\mathbb{R}}_-\neq\emptyset.$
\end{itemize}

Then, there exists an element $x_0 \in K$ such that
\begin{equation*}
F(x_0, y)\cap {\mathbb{R}}_-\neq\emptyset, \,\forall y\in K.
\end{equation*}
\end{thm}

Obviously these results hold not only in real normed spaces but also in Hausdorff topological vector spaces. The convexity of a set-valued map
$F:D\subseteq X\rightrightarrows{\mathbb{R}}$, where $X$ is Hausdorff topological vector space, is understood in sense that for all $%
x_1,x_2,\ldots, x_n\in D$ and $\lambda_i\ge 0,\, i\in\{1,2,\ldots,n\},\,
\sum_{i=1}^n \lambda_i=1$ such that $\sum_{i=1}^n \lambda_i x_i\in D$ one
has
\begin{equation}\label{e1}
\sum_{i=1}^n \lambda_i F(x_i)\subseteq F\left(\sum_{i=1}^n \lambda_i
x_i\right).
\end{equation}
Here the usual Minkowski sum of sets is meant by the summation sign. To define concavity in the same setting, one replaces the last inclusion by
\begin{equation}\label{e2}
\sum_{i=1}^n \lambda_i F(x_i)\supseteq F\left(\sum_{i=1}^n \lambda_i
x_i\right)
\end{equation}

Note that in the definition of these notions we do not assume that $D$ is convex. 

The classical single-valued the equilibrium problem (see \cite{Fan}) for $\varphi :K
\times K \longrightarrow {\mathbb{R}}$ consists in finding $x_0\in K$ such
that

\begin{equation*}
\varphi(x_0,y)\ge 0,\,\forall y\in K.
\end{equation*}

We recall the famous existence result of to Ky Fan.

\begin{thm}
\label{tKF} Let $K$ be a nonempty convex compact subset of the Hausdorff topological vector space $X$ and let $\varphi
:K \times K \longrightarrow{\mathbb{R}}$ be a function satisfying

\begin{itemize}
\item[(i)] $\forall y\in K$ the function $x\longrightarrow \varphi(x,y)$ is
upper semicontinuous on $K,$

\item[(ii)] $\forall x\in K,$ the function $y\to\varphi(x,y)$ is quasiconvex on $%
K,$

\item[(iii)] $\forall x\in D, \varphi(x, x)\ge 0.$

Then, there exists an element $x_0\in K$ such that
\begin{equation*}
\varphi(x_0,y)\ge 0,\,\forall y\in K.
\end{equation*}
\end{itemize}
\end{thm}

In subsequent sections, the notion of a KKM map and the well-known intersection Lemma due to Ky Fan (see \cite{Fan}) will be needed.

\begin{defn}
(Knaster-Kuratowski-Mazurkiewicz) Let $X$ be a Hausdorff topological vector
space and let $M\subseteq X.$ The application $G:M\rightrightarrows X$ is
called a KKM application if for every finite number of elements $%
x_1,x_2,\dots,x_n\in M$ one has $\co\{x_1,x_2,\ldots,x_n\}\subseteq %
\displaystyle\bigcup_{i=1}^n G(x_i).$
\end{defn}

\begin{lem}
Let $X$ be a Hausdorff topological vector space, $M\subseteq X$ and $%
G:M\rightrightarrows X$ be a KKM application. If $G(x)$ is closed for every $%
x\in M$, and there exists $x_{0}\in M,$ such that $G(x_{0})$ is compact,
then
\begin{equation*}
\bigcap_{x\in M}G(x)\neq \emptyset .
\end{equation*}
\end{lem}



\subsection{Self segment-dense sets}


Le $X$ be a  Hausdorff topological vector space. We will  use the following notations for the open, respectively
closed, line segments in $X$ with the endpoints $x$ and $y$
\begin{eqnarray*}
(x,y) &=&\big\{z\in X:z=x+t(y-x),\,t\in (0,1)\big\}, \\
\lbrack x,y] &=&\big\{z\in X:z=x+t(y-x),\,t\in \lbrack 0,1]\big\}.
\end{eqnarray*}

In \cite{DTL}, Definition 3.4, The Luc has introduced the notion of a
so-called \emph{segment-dense} set. Let $V\subseteq X$ be a convex set. One
says that the set $U\subseteq V$ is segment-dense in $V$ if for each $x\in V$
there can be found $y\in U$ such that $x$ is a cluster point of the set $%
[x,y]\cap U.$

In what follows we present a denseness notion (see also \cite{LaVi}) which is slightly different from the
concept of The Luc presented above, but which is better suited for our needs.

\begin{defn}
\label{dd} Consider the sets $U\subseteq V\subseteq X$ and assume that $V$
is convex.
We say that $U$ is self segment-dense in $V$ if $U$ is dense in $V$ and
\begin{equation*}
\forall x,y\in U,\mbox{  the set }\lbrack x,y]\cap U\mbox{  is dense in }%
\lbrack x,y].
\end{equation*}
\end{defn}

\begin{rmk}\textrm{ Obviously in one dimension the concepts of a segment-dense set respectively a self
segment-dense set are equivalent to the concept of a dense set. }
\textrm{\ }
\end{rmk}

In what follows we provide an essential example of a self segment-dense set.

\begin{ex}
\textrm{(see also \cite{LYAK}, Example 3.1)\label{EEE} Let $V$ be the two
dimensional Euclidean space ${\mathbb{R}}^2$ and define $U$ to be the set
\begin{equation*}
U :=\{(p,q) \in{\mathbb{R}}^2 : p\in \mathbb{Q},\, q\in\mathbb{Q}\},
\end{equation*}
where $\mathbb{Q}$ denotes the set of all rational numbers. Then, it is
clear that $U$ is dense in ${\mathbb{R}}^2.$ On the other hand $U$ is not
segment-dense in ${\mathbb{R}}^2,$ since for $x=(0, \sqrt{2})\in {\mathbb{R}}%
^2 $ and for every $y=(p,q)\in U$, one has $[x, y] \cap U = \{y\}.$ }

\textrm{It can easily be observed that $U$ is self segment-dense in ${%
\mathbb{R}}^2$, since for every $x,y\in U$ $x=(p,q),\,y=(r,s)$ we have $%
[x,y]\cap U=\{(p+t(r-p),q+t(s-q)): t\in[0,1]\cap\mathbb{Q}\},$ which is
obviously dense in $[x,y].$ }
\end{ex}

To further circumscribe the notion of a slef segment-dense set we provide an example %
of a subset that is dense but not self segment-dense.

\begin{ex}\label{ex1}
\textrm{Let $X$ be an infinite dimensional real Hilbert space, it is known that the unit sphere %
$\left\{ x\in X:\left\Vert x\right\Vert =1\right\} $
is dense with respect to the weak topology in the unit ball $\left\{ x\in X:\left\Vert x\right\Vert \leq1 \right\}$, but %
it is obviously not self segment-dense since any segment with endpoints on the sphere does not intersect the sphere in any other points.}
\end{ex}

\begin{rmk}
\textrm{Note that every dense convex subset of a Banach space is self
segment-dense. In particular dense subspaces and dense affine subsets are
self segment-dense. }
\end{rmk}


\section{Self segment-dense sets and equilibrium problems}


In this section, by making use of the concept of a self segment-dense set,
we obtain existence results for set-valued equilibrium problems. Ky Fan's
lemma is used in the proof of Theorem \ref{t31} and Theorem \ref{t511}, the main results of this
section, in order to establish the existence of solutions to equilibrium
problems. This approach is well known in the literature, see, for instance,
\cite{KV,DTL,L1,LYY,Y1,Y2}.

The following lemma gives an interesting characterization of self
segment-dense sets and will be used in the sequel. If $X$ is a Hausdorff locally convex topological vector space, %
then the origin has a local base of convex, balanced and absorbent
sets, and recall, that the set
\begin{equation*}
\core D=\{u\in D| \ \forall x\in X \ \exists \delta>0\mbox{ such that }%
\forall \epsilon\in[0,\delta]:u+\epsilon x\in D\}
\end{equation*}
is called the algebraic interior (or core) of $D\subseteq X$ (see \cite{Zal-carte}).

If $D$ is convex with nonempty interior, then $\inte(D)=\core(D)$ (see \cite%
{Zal-carte}).

\begin{lem}
\label{l51} Let $X$ be a Hausdorff locally convex topological vector space,
let $V\subseteq X$ be a convex set and let $U\subseteq V$ a self
segment-dense set in $V.$ Then for all finite subset $\{u_1,u_2,\ldots,u_n\}%
\subseteq U$ one has $$\cl(\co\{u_1,u_2,\ldots,u_n\}\cap U)=\co%
\{u_1,u_2,\ldots,u_n\}.$$
\end{lem}

\begin{proof} We prove the statement by classical induction. For $n=2$ by using the self segment-denseness of $U$ in $V$ we have that for every $u_1,u_2\in U$ $$\cl(\co\{u_1,u_2\}\cap U)=\cl([u_1,u_2]\cap U)=[u_1,u_2]=\co\{u_1,u_2\}.$$

Assume that the statement holds for every $u_1,u_2,\ldots, u_{n-1}\in U$ and we show that is also true for  all $u_1,u_2,\ldots, u_{n}\in U.$
For this let us fix $u_1,u_2,\ldots, u_{n-1}\in U$ and let $u_n\in U.$ Obviously one should take $u_n$ such that $\co\{u_1,u_2,\ldots,u_n\}\neq \co\{u_1,u_2,\ldots,u_{n-1}\}.$ In this case
$$\co\{u_1,u_2,\ldots,u_n\}=\bigcup_{u\in\co\{u_1,u_2,\ldots,u_{n-1}\}}[u,u_n].$$

We must show that $\co\{u_1,u_2,\ldots,u_n\}\cap U$  is  dense in $\co\{u_1,u_2,\ldots,u_n\}.$

Assume the contrary, that is, there exists $s\in \co\{u_1,u_2,\ldots,u_n\}$ and an neighbourhood $S$ of $s$ such that $S\cap  \co\{u_1,u_2,\ldots,u_n\}$ contains no points from $U.$ Obviously, we can take $S=s+G$, where $G$ is an open, balanced and convex neighbourhood of the origin. Note that we have $s=u_n+t(u-u_n)$ for some $t\in(0,1)$, $u\in \co\{u_1,u_2,\ldots,u_{n-1}\}.$

By the induction hypothesis, $\co\{u_1,u_2,\ldots,u_{n-1}\}\cap U$  is  dense in $\co\{u_1,u_2,\ldots,u_{n-1}\},$ hence,  there exists a net $(u^\a)\subseteq \co\{u_1,u_2,\ldots,u_{n-1}\}\cap U$ such that $\lim u^\a= u.$ Thus, for $u+G$ a neighbourhood of $u$ there exists $\a_0$ such that $u^\a\in u+G$ for all $\a>\a_0.$

We show next, that for $u^\a\in u+G$ we have $s^\a=u_n+t(u^\a-u_n)\in s+G.$

Indeed $(u^\a-u_n)\in (u-u_n)+G$ and since $G$ is balanced and $t\in(0,1)$ we have that $t(u^\a-u_n)\in t(u-u_n)+G$. Hence  $s^a=u_n+t(u^\a-u_n)\in u_n+t(u-u_n)+G= s+G.$

Note that $s+G$ is open and convex, hence $s+G=\core(s+G),$ which shows that $s^\a\in \core(s+G).$ Therefore, there exists  $\d>0$ such that $s^a+\e (u_n-u^\a)\in s+G,$ respectively $s^\a+\e (u^\a-u_n)\in s+G$ for all $\e\in [0,\d].$ Hence, $s^\a\in [s^\a+\e (u^\a-u_n),s^\a+\e (u_n-u^\a)]\subseteq [u^\a,u_n]$ for all $\e\in [0,\d].$ Since $u^\a,u_n\in U$ and $U$ is self segment-dense, obviously $[s^\a+\e (u^\a-u_n),s^\a+\e (u_n-u^\a)]\cap U\neq  \emptyset$ for all $\e\in (0,\d],$ which leads to $$(s+G)\cap \co\{u_1,u_2,\ldots, u_n\}\cap U\neq\emptyset,$$ which yields a contradiction.
\end{proof}

\begin{rmk}
\textrm{In Lemma \ref{l51}, the assumption that $U$ is self segment-dense
cannot be replaced by the denseness of $U$ as the next example shows. }
\end{rmk}

\begin{ex}
\textrm{Let $V$ be the unit ball in ${\mathbb{R}}^{3}$, let $A$ be the
interior of a square with vertices $(-1,0,0),$ $(0,-1,0),(1,0,0)$ respectively $%
(0,1,0).$ Then obviously $U=V\setminus A$ is dense in $V$ but not self
segment-dense in $V,$ since, for instance, for $u_{1}=(\frac{3}{5},\frac{3}{5%
},0)\in U$ and $u_{2}=(-\frac{3}{5},-\frac{3}{5},0)\in U$ the set $%
[u_{1},u_{2}]\cap U$ is not dense in $[u_{1},u_{2}].$ This also shows, that $%
\cl(\co\{u_{1},u_{2}\}\cap U)\neq \co\{u_{1},u_{2}\}.$ }
\end{ex}

The next result gives an important application of self segment-dense sets in
the framework of equilibrium problems presented above.

\begin{thm}
\label{t31} Let $X$ be a Hausdorff locally convex topological vector space,
let $K$ be a nonempty convex compact subset of $X,$ let $D\subseteq K$ be a
self segment-dense set, and let $F :K \times K \rightrightarrows {\mathbb{R}}
$ be a set valued map satisfying

\begin{itemize}
\item[(i)] $\forall y\in D, x\longrightarrow F(x, y)$ is lower
semicontinuous on $K$,

\item[(ii)] $\forall x\in K, y\longrightarrow F(x, y)$ is lower
semicontinuous on $K\setminus D$,

\item[(iii)] $\forall x\in D, y\longrightarrow F(x, y)$ is convex on $D$,

\item[(iv)] $\forall x\in D, F(x,x)\ge 0.$
\end{itemize}

Then, there exists an element $x_0 \in K$ such that
\begin{equation*}
F(x_0, y)\ge 0, \,\forall y\in K.
\end{equation*}
\end{thm}

\begin{proof}
Consider the map $$G:D\toto K,\, G(y)=\{x\in K:F(x,y)\ge 0\}.$$ We show that $\bigcap_{y\in  D}G(y)\neq\emptyset,$ or, in other words, that there exists $x_0\in K$ such that $$F(x_0,y)\ge 0,\,\forall y\in D.$$
We start by showing that $G(y)$ is closed for all $y\in D.$ To this end we fix $y\in D$ and consider the net $(x^\a)\subseteq G(y),\, \lim x^\a=x\in K.$
%
%
%
According to Remark \ref{r21}, from the lower semicontinuity assumption $(i)$ one has that
for every $x^*\in F(x,y)$ there exists a net $x^*_\a \in F(x^\a,y)$ such that $x^*_\a\longrightarrow x^*.$ On the other hand $x^*_\a\ge 0$ for all $\a$, hence $x^*\ge 0.$ Thus, $F(x,y)\ge 0$ which shows that   $x\in G(y)$ and the set $G(y)\subseteq K$ is closed.

Since $K$ is compact, we also have that $G(y)$ is compact for all $y\in D.$

Next, we show that $G$ is a KKM mapping. In fact, we prove by a contradiction argument that that given arbitrary $y_1,y_2,\ldots,y_n\in D,$ $$\co\{y_1,y_2,\ldots,y_n\}\cap D\subseteq\bigcup_{i=1}^n G(y_i).$$
So, assume the contrary, that there exist $\l_1,\l_2,\ldots,\l_n\ge 0,\,\sum_{i=1}^n\l_i=1$ such that $\sum_{i=1}^n\l_i y_i\in D$ and $$\sum_{i=1}^n\l_i y_i\not\in\bigcup_{i=1}^n G(y_i).$$
This is equivalent with $F(\sum_{i=1}^n\l_i y_i,y_i)\cap (-\infty,0)\neq \emptyset,\,\forall i\in\{1,2,\ldots,n\},$ and hence,
$$\sum_{i=1}^n \l_iF\left(\sum_{i=1}^n\l_i y_i,y_i\right)\cap (-\infty,0)\neq \emptyset.$$  Since by assumption $(iii)$ $\forall x\in D$ the mapping $y\To F(x, y)$ is convex on $D$ we have $$\sum_{i=1}^n \l_iF\left(\sum_{i=1}^n\l_i y_i,y_i\right)\subseteq F\left(\sum_{i=1}^n\l_i y_i, \sum_{i=1}^n\l_i y_i\right )\ge0,$$
 or equivalently 
 $$\sum_{i=1}^n \l_iF\left(\sum_{i=1}^n\l_i y_i,y_i\right)\subseteq [0,\infty),$$ which contradicts our initial assumption.
Consequently, $$\co\{y_1,y_2,\ldots,y_n\}\cap D\subseteq\bigcup_{i=1}^n G(y_i),$$ holds true, and leads to
 $$\cl(\co\{y_1,y_2,\ldots,y_n\}\cap D)\subseteq\cl\left(\bigcup_{i=1}^n G(y_i)\right).$$

Furthermore, since $G(y_i)$ is closed for all $i\in\{1,2,\ldots,n\}$ we have
$$\cl\left(\bigcup_{i=1}^n G(y_i)\right)=\bigcup_{i=1}^n G(y_i).$$
On the other hand, according to Lemma \ref{l51} $\cl(\co\{y_1,y_2,\ldots,y_n\}\cap D)=\co\{y_1,y_2,\ldots,y_n\},$ so
 $$\co\{y_1,y_2,\ldots,y_n\}\subseteq\bigcup_{i=1}^n G(y_i).$$ Hence, $G$ is a KKM map.

Thus, according to Ky Fan's lemma $\bigcap_{y\in  D}G(y)\neq\emptyset.$ In other words, there exists $x_0\in K$ such that $F(x_0,y)\ge 0$ for all $y\in D.$

At this point we make use of the assumption $(ii)$ to extend the previous statement to the whole set $K$. Consider $y\in K\setminus D,$ since $D$ is dense in $K$, there exists a net $(y^\a)\subseteq D$ such that $\lim y^a=y.$ Now, due to the assumption $(ii)$ and Remark \ref{r21}, for every $y^*\in F(x_0,y)$ there exists a net $y^*_\a\in F(x_0,y^\a)$ such that $\lim y^*_\a=y^*.$ But obviously $y^*_\a\ge 0$, hence $y^*\ge 0,$ and finally $F(x_0,y)\ge 0,\,\forall y\in K.$
\end{proof}

In the above Theorem, one can replace $F$ by $-F$ and obtain a
result concerning the opposite inequalities.

\begin{thm}
\label{rt51} Let $X$ be a Hausdorff locally convex topological vector space,
let $K$ be a nonempty convex compact subset of $X,$ let $D\subseteq K$ be a
self segment-dense set, and let $F :K \times K \rightrightarrows {\mathbb{R}}
$ be a set valued map satisfying

\begin{itemize}
\item[(i)] $\forall y\in D, x\longrightarrow F(x, y)$ is lower
semicontinuous on $K$,

\item[(ii)] $\forall x\in K, y\longrightarrow F(x, y)$ is lower
semicontinuous on $K\setminus D$,

\item[(iii')] $\forall x\in D, y\longrightarrow F(x, y)$ is convex on $D$,

\item[(iv')] $\forall x\in D, F(x,x)\le 0.$
\end{itemize}

Then, there exists an element $x_0 \in K$ such that
\begin{equation*}
F(x_0, y)\le 0, \,\forall y\in K.
\end{equation*}
\end{thm}

By similar methods to those used in the proof of Theorem \ref{t31} one can
obtain a result concerning the second set-valued equilibrium problem. The
following theorem holds.

\begin{thm}
\label{t511} Let $X$ be a Hausdorff locally convex topological vector space,
let $K$ be a nonempty convex compact subset of $X,$ let $D\subseteq K$ be a
self segment-dense set, and let $F :K \times K \rightrightarrows {\mathbb{R}}
$ be a set valued map satisfying

\begin{itemize}
\item[(i)] $\forall y\in D, x\longrightarrow F(x, y)$ is upper
semicontinuous on $K$,

\item[(ii)] $\forall x\in K, y\longrightarrow F(x, y)$ is upper
semicontinuous on $K\setminus D$,

\item[(iii)] $\forall x\in D, y\longrightarrow F(x, y)$ is concave on $D$,

\item[(iv)] $\forall x\in D, F(x,x)\cap {\mathbb{R}}_+\neq\emptyset.$
\end{itemize}

Then, there exists an element $x_0 \in K$ such that
\begin{equation*}
F(x_0, y)\cap {\mathbb{R}}_+\neq\emptyset, \,\forall y\in K.
\end{equation*}
\end{thm}

\begin{proof}Consider the map $G:D\toto K,\, G(y)=\{x\in K:F(x,y)\cap\Real_+\neq\emptyset\}.$ We show that $G(y)$ is closed for all $y\in D.$ Indeed, for a fixed $y\in D$ we have $G(y)=f_y^{-1}(\Real_+),$ where $f_y:K\toto\Real,\, f_y(x)=F(x,y).$ From (i) we have that $f_y$ is upper semicontinuous on $K$ and since $\Real_+$ is closed, according to Remark \ref{r21} $f_y^{-1}(\Real_+)$ is closed. Hence $G(y)\subseteq K$ is closed for all $y\in D$, and by the compactness of $K$ we get that $G(y)$ is compact for every $y\in D.$

Following the proof of Theorem \ref{t31}, it can be  shown  that $\bigcap_{y\in  D}G(y)\neq\emptyset,$ that is, there exists $x_0\in K$ such that $$F(x_0,y)\cap\Real_+\neq\emptyset,\,\forall y\in D.$$
Now, let us fix $y\in K\setminus D$ and assume that $F(x_0,y)\in(-\infty,0).$ Since the set-valued function $F(x_0, .)$ is upper semicontinuous at $y$ we obtain that there exists an open neighbourhood $U$ of $y$, such that for all $F(x_0,U)\subseteq (-\infty,0).$ But $D$ is dense in $K$, and hence there exists $u\in U$ such that $u\in D$, so $F(x_0,u)\cap\Real_+\neq \emptyset,$ a contradiction. Thus, $$F(x_0,y)\cap\Real_+,\,\forall y\in K$$ must hold true.
\end{proof}


The reminder of this section is concerned with the single-valued equilibrium
problem.

Let $K\subseteq X$ be a subset and let $f:K\longrightarrow{\mathbb{R}}.$ We
say that $f$ is convex on $K,$ respectively concave on $K$, if for all $x_1,x_2,\ldots x_n\in K$ and $\lambda_i\ge 0,\,
i\in\{1,2,\ldots,n\},\, \sum_{i=1}^n \lambda_i=1$ such that $\sum_{i=1}^n
\lambda_i x_i\in K$ one has
\begin{equation*}
\sum_{i=1}^n \lambda_i f(x_i)\ge f\left(\sum_{i=1}^n \lambda_i x_i\right),%
\mbox{ respectively }\sum_{i=1}^n \lambda_i f(x_i)\le f\left(\sum_{i=1}^n
\lambda_i x_i\right).
\end{equation*}

Note that in these definitions we do not assume the convexity of $K.$ We have the following existence result for the single valued equilibrium
problem.

\begin{thm}
\label{t52} Let $X$ be a Hausdorff locally convex topological vector space,
let $K$ be a nonempty convex compact subset of $X$, let $D\subseteq K$ be a
self segment-dense set and let $\varphi :K \times K \longrightarrow{\mathbb{R%
}}$ a function satisfying

\begin{itemize}
\item[(i)] $\forall y\in D$ the function $x\longrightarrow \varphi(x,y)$ is
upper semicontinuous on $K,$

\item[(ii)] $\forall x\in K, y\longrightarrow \varphi(x, y)$ is upper
semicontinuous on $K\setminus D$,

\item[(iii)] $\forall x\in D,$ the function $y\to\varphi(x,y)$ is convex on $%
D,$

\item[(iv)] $\forall x\in D, \varphi(x, x)\ge 0.$
\end{itemize}

Then, there exists an element $x_0\in K$ such that
\begin{equation*}
\varphi(x_0,y)\ge 0,\,\forall y\in K.
\end{equation*}
\end{thm}

\begin{proof} We give only an outline of the proof since the ideas are similar to those used in the proof of Theorem \ref{t31}.

We consider the map $$G:D\toto K,\,\, G(y)=\{x\in K:\varphi(x,y)\ge 0\}.$$ Observe that for a fixed $y\in D$ the set  $G(y)$ is the superlevel set $\{x\in K:f_y(x)\ge 0\}$ of the function $f_y:K\to\Real,\, f_y(x)=\varphi(x,y).$
Due to the assumption $(i)$, we have that $G(y)$ is closed for all $y\in D.$

Further, from assumptions $(iii)$, $(iv)$  and Lemma \ref{l51} we obtain that $G$ is a KKM application. Then, according to Ky Fan's lemma $\bigcap_{y\in  D}G(y)\neq\emptyset.$ So, there exists $x_0\in K$ such that $\varphi(x_0,y)\ge 0$ for all $y\in D.$

Finally, if $y\in K\setminus D$ by the denseness of $D$ in $K$ there exists a net $(y^\a)\subseteq D$ such that $\lim y^a=y.$ At this point, the assumption $(ii)$, $\varphi(x_0, y)$ the upper semicontinuity of $\varphi(x_0, y)$ on $K\setminus D$, assures that  $0\le\limsup_{y^\a\to y} \varphi(x_0,y^\a)\le \varphi(x_0,y).$ Thus we have $\varphi(x_0,y)\ge 0$ for all $y\in K.$

\end{proof}
The above result has also a complementary formulation in which convexity is
replaced by concavity and the inequalities have opposite direction.

\begin{thm}
\label{c51} Let $X$ be a Hausdorff locally convex topological vector space,
let $K$ be a nonempty convex compact subset of $X$, let $D\subseteq K$ be a
self segment-dense set and let $\varphi :K \times K \longrightarrow{\mathbb{R%
}}$ a function satisfying

\begin{itemize}
\item[(i)] $\forall y\in D$ the function $x\longrightarrow \varphi(x,y)$ is
lower semicontinuous on $K,$

\item[(ii)] $\forall x\in K, y\longrightarrow \varphi(x, y)$ is lower
semicontinuous on $K\setminus D$,

\item[(iii)] $\forall x\in D,$ the function $y\to\varphi(x,y)$ is concave on
$D,$

\item[(iv)] $\forall x\in D, \varphi(x, x)\le 0.$
\end{itemize}

Then, there exists an element $x_0\in K$ such that
\begin{equation*}
\varphi(x_0,y)\le 0,\,\forall y\in K.
\end{equation*}
\end{thm}

\begin{proof} Apply Theorem \ref{t52} to the function $-\varphi.$
\end{proof}

In what follows we show that the assumption that $D$ is self segment-densene, in the hypotheses of the previous theorems is essential, and it cannot be replaced by the denseness of $D.$

Indeed, let us consider the Hilbert space of square-summable sequences $l_2$, and let $K=\{x\in l_2:\|x\|\le 1\}$ be its unit ball while $D=\{x\in l_2:\|x\|= 1\}$ is the unit sphere. It is well known that $l_2$ endowed with the weak topology is a Hausdorff locally convex topological vector space, and by Banach-Alaoglu theorem $K$ is compact in this topology. Further, we have seen in Example \ref{ex1} that $D$ is dense, but not self segment-dense in $K.$

In this setting we define the single-valued map
$$\varphi:K\times K\to\Real,\, \varphi(x,y)=\<x,y\>-1,$$
which has the following properties:

\begin{itemize}
\item[(a)] for all $y\in K, x\longrightarrow \varphi(x,y)$ is continuous on $K,$

\item[(b)] for all $x\in K, y\longrightarrow \varphi(x,y)$ is continuous on $K,$

\item[(c)] for all $x\in K, y\longrightarrow\varphi(x,y)$ is affine, hence convex and also concave on $K,$

\item[(d)] $\varphi(x,x)=0$ for all $x\in D.$

\end{itemize}

Now, consider the operators $F_1,F_2:K\times K\toto\Real,$

$$F_1(x,y)=[\varphi(x,y),\infty),$$
and
$$F_2(x,y)=(-\infty,\varphi(x,y)].$$

Obviously $F_1(x,x)=[0,\infty)$ for all $x\in D$, while $F_2(x,x)=(-\infty,0]$ for all $x\in D,$ and it can easily be shown that $F_1$, respectively $F_2$ satisfy the conditions (i)-(iii) in Theorem \ref{t31}, respectively Theorem \ref{t511} (even some stronger assumptions, since we can take everywhere $x,y\in K$).

S we see that $F_1,\,F_2$, respectively $\varphi$ satisfy all the assumptions of Theorem \ref{t31}, Theorem \ref{t511}, respectively Theorem \ref{t52} except the asspumtion that $D$ is self segment-dense (here $D$ is only dense) and also that the conclusions of the above mentioned theorems fail, since for $y=0\in K$ one has
$$F_1(x,y)=[-1,\infty),\,\forall x\in K,$$
$$F_2(x,y)=(-\infty,-1],\,\forall x\in K,$$
respectively
$$\varphi(x,y)=-1,\,\forall x\in K.$$


\section{Densely defined equilibrium problems without compactness assumptions}


The compactness of the domain $K$ in the hypotheses of the existence theorems in Section 3 is a rather strong condition, so a natural question is whether similar existence results can be obtained without a compactness assumption. In this context, one can observe that the KKM mappings built in the proofs of the above mentioned results are compact valued, while Ky Fan's lemma requires the existence of a single point where the KKM must be compact valued. Motivated by this observation, in what follows we replace the compactness assumption by the closedness of the domain $K$ in order to obtain existence results for the equilibrium problems presented so far.

\begin{thm}
\label{t36} Let $X$ be a Hausdorff locally convex topological vector space,
let $K$ be a nonempty convex closed subset of $X,$ let $D\subseteq K$ be a
self segment-dense set, and let $F :K \times K \rightrightarrows {\mathbb{R}}
$ be a set valued map satisfying

\begin{itemize}
\item[(i)] $\forall y\in D, x\longrightarrow F(x, y)$ is lower
semicontinuous on $K$,

\item[(ii)] $\forall x\in K, y\longrightarrow F(x, y)$ is lower
semicontinuous on $K\setminus D$,

\item[(iii)] $\forall x\in D, y\longrightarrow F(x, y)$ is convex on $D$,

\item[(iv)] $\forall x\in D, F(x,x)\ge 0.$

\item[(v)] $\exists K_0\subseteq X$ compact and $y_0\in D\cap K_0$ such that $F(x,y_0)\cap(-\infty,0)\neq\emptyset,\,\forall x\in K\setminus K_0.$
\end{itemize}

Then, there exists an element $x_0 \in K$ such that
\begin{equation*}
F(x_0, y)\ge 0, \,\forall y\in K.
\end{equation*}
\end{thm}

\begin{proof}
Consider the map $$G:D\toto K,\, G(y)=\{x\in K:F(x,y)\ge 0\}.$$ According to the proof of Theorem \ref{t52}, $G(y)$ is closed for all $y\in D.$ We show that $G(y_0)$ is compact and the rest of the proof is similar to the proof of Theorem \ref{t52}. It is thus enough to show that $G(y_0)\subseteq K_0.$
 Assume the contrary, that there exists $z\in G(y_0)$ such that $z\not\in K_0.$ Then $F(z,y_0)\ge 0$ which contradicts (v).
\end{proof}

The following results can be proved analogously.

\begin{thm}
\label{t37} Let $X$ be a Hausdorff locally convex topological vector space,
let $K$ be a nonempty convex closed subset of $X,$ let $D\subseteq K$ be a
self segment-dense set, and let $F :K \times K \rightrightarrows {\mathbb{R}}
$ be a set valued map satisfying

\begin{itemize}
\item[(i)] $\forall y\in D, x\longrightarrow F(x, y)$ is upper
semicontinuous on $K$,

\item[(ii)] $\forall x\in K, y\longrightarrow F(x, y)$ is upper
semicontinuous on $K\setminus D$,

\item[(iii)] $\forall x\in D, y\longrightarrow F(x, y)$ is concave on $D$,

\item[(iv)] $\forall x\in D, F(x,x)\cap {\mathbb{R}}_+\neq\emptyset.$

\item[(v)] $\exists K_0\subseteq X$ compact and $y_0\in D\cap K_0$ such that $F(x,y_0)\cap \mathbb{R}_+=\emptyset,\,\forall x\in K\setminus K_0.$
\end{itemize}

Then, there exists an element $x_0 \in K$ such that
\begin{equation*}
F(x_0, y)\cap {\mathbb{R}}_+\neq\emptyset, \,\forall y\in K.
\end{equation*}
\end{thm}


\begin{thm}
\label{t38} Let $X$ be a Hausdorff locally convex topological vector space,
let $K$ be a nonempty convex closed subset of $X$, let $D\subseteq K$ be a
self segment-dense set and let $\varphi :K \times K \longrightarrow{\mathbb{R%
}}$ a function satisfying

\begin{itemize}
\item[(i)] $\forall y\in D$ the function $x\longrightarrow \varphi(x,y)$ is
upper semicontinuous on $K,$

\item[(ii)] $\forall x\in K, y\longrightarrow \varphi(x, y)$ is upper
semicontinuous on $K\setminus D$,

\item[(iii)] $\forall x\in D,$ the function $y\to\varphi(x,y)$ is convex on $%
D,$

\item[(iv)] $\forall x\in D, \varphi(x, x)\ge 0.$

\item[(v)] $\exists K_0\subseteq X$ compact and $y_0\in D\cap K_0$ such that $\varphi(x,y_0)<0,\,\forall x\in K\setminus K_0.$
\end{itemize}

Then, there exists an element $x_0\in K$ such that
\begin{equation*}
\varphi(x_0,y)\ge 0,\,\forall y\in K.
\end{equation*}
\end{thm}

The condition (v) in Theorem \ref{t36}, Theorem \ref{t37} respectively Theorem \ref{t38} seem to be not so easy to verify, however, it is well known  that in a reflexive Banach space $X$,  the closed ball $\overline{B}_r=\{x\in X:\|x\|\le r\},r>0,$ is weakly compact. Therefore, if we endow the reflexive Banach space $X$ with the weak topology, condition (v) in the hypotheses of the previous theorems becomes :
\begin{itemize}
\item[(v')] $\exists r>0$ and $y_0\in D,\,\|y_0\|\le r$ such that for all $x\in K,\, \|x\|>r$ $F(x,y_0)\cap (-\infty,0)\neq\emptyset$ holds.
\item[(v'')] $\exists r>0$ and $y_0\in D,\,\|y_0\|\le r$ such that for all $x\in K,\, \|x\|>r$ $F(x,y_0)\cap \Real_+=\emptyset$ holds.
\item[(v''')] $\exists r>0$ and $y_0\in D,\,\|y_0\|\le r$ such that for all $x\in K,\, \|x\|>r$ $\varphi(x,y_0)<0$ holds.
\end{itemize}
Furthermore, in this setting condition (v) in the hypotheses of Theorem \ref{t36}, \ref{t37} and \ref{t38} can be weakened by assuming that  $\exists r>0$ such that for all $x\in K,\, \|x\|>r$ there exists $y_0\in K$ with $\|y_0\|<\|x\|$ and the appropriate condition
\begin{itemize}
\item[(i)] $F(x,y_0)\cap (-\infty,0)\neq\emptyset$,
\item[(ii)] $F(x,y_0)\cap \Real_+=\emptyset$,
\item[(iii)] $\varphi(x,y_0)<0$
\end{itemize}
holds.

More precisely we have the following result.

\begin{thm}
\label{t39} Let $X$ be a reflexive Banach space, let $K$ be a nonempty convex closed subset of $X,$ let $D\subseteq K$ be a
self segment-dense set in the weak topology of $X$, and let $F :K \times K \rightrightarrows {\mathbb{R}}
$ be a set valued map satisfying

\begin{itemize}
\item[(i)] $\forall y\in D, x\longrightarrow F(x, y)$ is weak lower
semicontinuous on $K$,

\item[(ii)] $\forall x\in K, y\longrightarrow F(x, y)$ is weak lower
semicontinuous on $K\setminus D$,

\item[(iii)] $\forall x\in K, y\longrightarrow F(x, y)$ is convex on $K$,

\item[(iv)] $\forall x\in K, F(x,x)\ge 0,\mbox{ and }F(x,x)\supseteq\{ 0\}.$

\item[(v)] $\exists r>0$ such that for all $x\in K,\, \|x\|>r$ there exists $y_0\in K$ with $\|y_0\|<\|x\|$  such that $F(x,y_0)\cap(-\infty,0]\neq\emptyset.$
\end{itemize}

Then, there exists an element $x_0 \in K$ such that
\begin{equation*}
F(x_0, y)\ge 0, \,\forall y\in K.
\end{equation*}
\end{thm}
\begin{proof} Let $r>0$ such that (v) holds and let $r_1>r.$ Let $K_0=K\cap \overline{B}_{r_1}.$ Since $K$ is convex and closed it is also weakly closed, $\overline{B}_{r_1}$ is weakly compact, hence $K_0$ is convex and weakly compact. According to Theorem \ref{t31}, there exists $x_0\in K_0$ such that $F(x_0,y)\ge 0$ for all $y\in K_0.$

Next, we show that there exists $z_0\in K_0,\|z_0\|<r_1$ such that $F(x_0,z_0)\supseteq\{0\}.$ Indeed, if $\|x_0\|<r_1$ then let $z_0=x_0$ and the conclusion follows by (iv). If $\|x_0\|=r_1>r$ then by (v) we have that there exists $z_0\in K$ with $\|z_0\|<\|x_0\|$  such that $F(x_0,z_0)\cap(-\infty,0]\neq\emptyset.$ On the other hand, since $z_0\in K_0$ we have $F(x_0,z_0)\ge 0$, hence $\{0\}\subseteq F(x_0,z_0).$

Let $y\in K.$ Then there exists $\l\in[0,1]$ such that $\l z_0+(1-\l)y\in K_0.$ Therefore $F(x_0,\l z_0+(1-\l)y)\ge 0$ and by (iii) we obtain
$$\l F(x_0,z_0)+(1-\l)F(x_0,y)\subseteq F(x_0,\l z_0+(1-\l)y)\subseteq[0,\infty).$$ Since $\{0\}\subseteq F(x_0,z_0)$ we have $F(x_0,y)\subseteq [0,\infty).$
\end{proof}

Similar results can be obtained for the other two equilibrium problems studied in this paper. However, if one compares Theorem \ref{t39} with Theorem \ref{t31} or Theorem \ref{t36} one observes that conditions (iii) and (iv) have been considerable changed. This is due the fact that condition (v) in Theorem \ref{t39} with the assumptions (iii) and (iv) of Theorem \ref{t31} or Theorem \ref{t36} does not assure the existence of a solution when $K$ is  closed but not compact.

Our purpose is to overcome this situation by replacing (v) with a condition that assures the existence of a solution under the original assumptions (iii) and (iv). In fact, we show that if $\forall x\in K, y\longrightarrow F(x, y)$ is convex on $D,$ respectively $\forall x\in D, F(x,x)\ge 0,$ instead of (iii), respectively (iv) in the previous theorem, then we can replace (v) with:

$\exists r>0$ such that for all $x\in K,\, \|x\|\le r$ there exists $y_0\in D$ with $\|y_0\|<r$  such that $\{0\}\subseteq F(x,y_0).$

The following result holds.
\begin{thm}
\label{t391} Let $X$ be a reflexive Banach space, let $K$ be a nonempty convex closed subset of $X,$ let $D\subseteq K$ be a
self segment-dense set in the weak topology of $X$, and let $F :K \times K \rightrightarrows {\mathbb{R}}
$ be a set valued map satisfying

\begin{itemize}
\item[(i)] $\forall y\in D, x\longrightarrow F(x, y)$ is weak lower
semicontinuous on $K$,

\item[(ii)] $\forall x\in K, y\longrightarrow F(x, y)$ is weak lower
semicontinuous on $K\setminus D$,

\item[(iii)] $\forall x\in K, y\longrightarrow F(x, y)$ is convex on $D$,

\item[(iv)] $\forall x\in D, F(x,x)\ge 0.$

\item[(v)] $\exists r>0$ such that for all $x\in K,\, \|x\|\le r$ there exists $y_0\in D$ with $\|y_0\|<r$  such that $\{0\}\subseteq F(x,y_0).$
\end{itemize}

Then, there exists an element $x_0 \in K$ such that
\begin{equation*}
F(x_0, y)\ge 0, \,\forall y\in K.
\end{equation*}
\end{thm}
\begin{proof} Let $r>0$ such that (v) holds and consider the weakly compact set $K_0=K\cap \overline{B}_r.$
According to Theorem \ref{t31} there exists $x_0\in K_0$ such that $F(x_0, y)\ge 0, \,\forall y\in K_0.$

According to (v), there exists $z_0\in D$ with $\|z_0\|<r$  such that $\{0\}\subseteq F(x,z_0).$

Now let $z\in D\setminus K_0.$ Then, in virtue of self segment denseness of $D$ in $K$, there exists $\l\in(0,1)$ such that $\l z_0+(1-\l)z\in K_0\cap D.$ According to (iii)
$$F(x_0,\l z_0+(1-\l)z)\supseteq  \l F(x_0,z_0)+(1-\l)F(x_0,z),$$
but $F(x_0,\l z_0+(1-\l)z)\ge 0$ and $\{0\}\subseteq F(x,z_0)$ which leads to $F(x_0,z)\ge 0.$

Hence $F(x_0,z)\ge 0$ for all $z\in D.$ Let $y\in K\setminus D$.  Since $D$ is dense in $K$ there exists a net $y^\a\subseteq D$ such that $\lim y^\a=y$ where the limit is taken in the weak topology of $X.$ According to (ii) $F(x_0,\cdot)$ is weakly lower semicontinuous at $y$. Now, due to Remark \ref{r21}, for every $y^*\in F(x_0,y)$ there exists a net $y^*_\a\in F(x_0,y^\a)$ such that $\lim y^*_\a=y^*.$ But obviously $y^*_\a\ge 0$, hence $y^*\ge 0,$ and finally $F(x_0,y)\ge 0,\,\forall y\in K.$
\end{proof}

Concerning the weaker set-valued equilibrium problem a similar result holds.
\begin{thm}\label{t310} Let $X$ be a reflexive Banach space, let $K$ be a nonempty convex closed subset of $X,$ let $D\subseteq K$ be a
self segment-dense set in the weak topology of $X$, and let $F :K \times K \rightrightarrows {\mathbb{R}}
$ be a set valued map satisfying
\begin{itemize}
\item[(i)] $\forall y\in D, x\longrightarrow F(x, y)$ is weak upper
semicontinuous on $K$,

\item[(ii)] $\forall x\in K, y\longrightarrow F(x, y)$ is weak upper
semicontinuous on $K\setminus D$,

\item[(iii)] $\forall x\in K, y\longrightarrow F(x, y)$ is concave on $D$,

\item[(iv)] $\forall x\in D, F(x,x)\cap\Real_+\neq\emptyset.$

\item[(v)] $\exists r>0$ such that for all $x\in K,\, \|x\|\le r$ there exists $y_0\in D$ with $\|y_0\|< r$  and $F(x,y_0)\le 0.$
\end{itemize}

Then, there exists an element $x_0 \in K$ such that
\begin{equation*}
F(x_0, y)\cap {\mathbb{R}}_+\neq\emptyset, \,\forall y\in K.
\end{equation*}
\end{thm}
\begin{proof} Let $r>0$ such that (v) holds and consider the weakly compact set $K_0=K\cap \overline{B}_r.$
According to Theorem \ref{t511} there exists $x_0\in K_0$ such that $F(x_0, y)\cap {\mathbb{R}}_+\neq\emptyset, \,\forall y\in K_0.$

According to (v), there exists $z_0\in D$ with $\|z_0\|<r$  such that $F(x,z_0)\le 0.$

Now let $z\in D\setminus K_0.$ Then, in virtue of self segment denseness of $D$ in $K$, there exists $\l\in(0,1)$ such that $\l z_0+(1-\l)z\in K_0\cap D.$ According to (iii)
$$F(x_0,\l z_0+(1-\l)z)\subseteq  \l F(x_0,z_0)+(1-\l)F(x_0,z),$$
which leads to $F(x_0,z)\cap {\mathbb{R}}_+\neq\emptyset.$

Hence $F(x_0,z)\cap {\mathbb{R}}_+\neq\emptyset$ for all $z\in D.$ Let $y\in K\setminus D$. According to (ii) $F(x_0,\cdot)$ is weakly upper semicontinuous at $y$, hence, for any open set $V\subseteq\Real$ there exists an open neighborhood $U$ of $y$ such that for any $u\in U$ one has $F(x_0,u)\subseteq V.$ Since $D$ is dense in $K$ we have $D\cap U\neq\emptyset.$ Assume that $F(x_0,y)\subseteq (-\infty,0).$ Then take $V=(-\infty,0)$ and let $z\in U\cap D.$ Then $F(x_0,z)\subseteq (-\infty,0)$ which contradicts the fact that $F(x_0,z)\cap {\mathbb{R}}_+\neq\emptyset.$
\end{proof}

For sake of completeness we also give sufficient conditions for the solution existence of densely defined single valued equilibrium problem in a reflexive Banach space setting.

\begin{thm}
\label{t44}  Let $X$ be a reflexive Banach space, let $K$ be a nonempty convex closed subset of $X,$ let $D\subseteq K$ be a
self segment-dense set in the weak topology of $X$, and let $\varphi :K \times K \longrightarrow{\mathbb{R%
}}$ a function satisfying

\begin{itemize}
\item[(i)] $\forall y\in D$ the function $x\longrightarrow \varphi(x,y)$ is
weak upper semicontinuous on $K,$

\item[(ii)] $\forall x\in K, y\longrightarrow \varphi(x, y)$ is weak upper
semicontinuous on $K\setminus D$,

\item[(iii)] $\forall x\in K,$ the function $y\to\varphi(x,y)$ is convex on $D,$

\item[(iv)] $\forall x\in D, \varphi(x, x)\ge 0.$

\item[(v)] $\exists r>0$ such that for all $x\in K,\, \|x\|\le r$ there exists $y_0\in D$ with $\|y_0\|< r$  and $\varphi(x,y_0)= 0.$

\end{itemize}

Then, there exists an element $x_0\in K$ such that
\begin{equation*}
\varphi(x_0,y)\ge 0,\,\forall y\in K.
\end{equation*}
\end{thm}
\begin{proof} Let $r>0$ such that (v) holds and consider the weakly compact set $K_0=K\cap \overline{B}_r.$
According to Theorem \ref{t52} there exists $x_0\in K_0$ such that $\varphi(x_0,y)\ge 0,\,\forall y\in K_0.$ According to (v), there exists $z_0\in D\cap K_0,\,\|z_0\|<r$ such that $\varphi(x_0,z_0)=0.$ Consider $z\in D\setminus K_0$ Since $D$ is self segment dense in $K$, there exists $\l\in(0,1)$ such that $\l z_0+(1-\l)z\in D\cap K_0.$
Hence, from (iii) we have $\varphi(x_0,\l z_0+(1-\l)z)\le \l\varphi(x_0,z_0)+(1-\l)\varphi(x_0,z),$ or equivalently
$(1-\l)\varphi(x_0,z)\ge \varphi(x_0,\l z_0+(1-\l)z)\ge 0.$ The latter relation shows, that $\varphi(x_0,z)\ge 0$ for all $z\in D.$

Finally, if $y\in K\setminus D$ by the denseness of $D$ in $K$ there exists a net $(y^\a)\subseteq D$ such that $\lim y^a=y$ where the limit is taken in the weak topology of $X.$  At this point, the assumption (ii), $\varphi(x_0, y)\ge 0$ for all $y\in D$ and  the upper semicontinuity of $\varphi(x_0, y)$ on $K\setminus D$, assures that  $$0\le\limsup_{y^\a\to y} \varphi(x_0,y^\a)\le \varphi(x_0,y).$$ Thus we have $\varphi(x_0,y)\ge 0$ for all $y\in K.$

\end{proof}


\section{A generalized Debreu-Gale-Nika\"{\i}do theorem}


As an application of the set-valued equilibrium results in the previous sections we
present a Debreu-Gale-Nika\"{\i}do-type theorem, which extends the famous, classical result in economic equilibrium theory by requiring that the collective Walras law holds not on the entire price simplex, but on a self segment-dense subset of it.  For the original results we refer to \cite{Debreu} Section 5.6(1), the Principal Lemma in \cite{Gale} and to Theorem 16.6 in \cite{Nikaido}.

Consider the simplex
\[
M^{n}=\left\{ x\in \mathbb{R}_{+}^{n}:\sum_{i=1}^{n}x_{i}=1\right\}
\]%
and the set valued map $C:M^{n}\rightrightarrows \mathbb{R}^{n}.$ Assume that $G(C)$, the  graph of $C$, is closed and $C$ has nonempty, bounded and convex values. According to Debreu-Gale-Nika\"{\i}do theorem, if for all $(x,y)\in G(C)$ we have $\<y,x\>\ge 0$ (Walras law), then there exists $x_0\in M^n$ such that
$$C(x_0)\cap \mathbb{R}_{+}^{n}\neq\emptyset.$$

In what follows we extend this result by weakening the conditions imposed on $C$ and assuming that Walras' law holds only on $D$, a self segment-dense subset of $M^{n}$.
Hence, we consider the set-valued map with nonempty compact convex values $C:M^{n}%
\rightrightarrows \mathbb{R}^{n}$. We will use the following notation%
\[
\forall y\in \mathbb{R}^{n},\quad \sigma \left( C\left( x\right) ,y\right)
:=\sup_{z\in C\left( x\right) }\left\langle z,y\right\rangle
\]%
and we say that $C$ is \emph{upper hemi-continuous} if the map $x\rightarrow
\sigma \left( C\left( x\right) ,y\right) $ is upper semi-continuous for all $%
y\in \mathbb{R}^{n}$. We have the following result.

\begin{thm}
Let $C:M^{n}\rightrightarrows \mathbb{R}^{n}$ be a set-valued map with
nonempty compact convex values and $D$ a self segment-dense subset of $M^{n}$. If

\begin{description}
\item[(i)] $C$ is upper hemi-continuous regarding $D$, i.e. $\forall y\in D,\quad$ $x\longrightarrow \sigma \left( C\left( x\right),y\right) $ is upper semi-continuous on $M^{n}.$

\item[(ii)] $\forall x\in D,\quad \sigma \left( C\left( x\right) ,x\right)
\geq 0.$
\end{description}

Then there exists $x_{0}\in M^{n}$ such that $C\left( x_{0}\right) \cap
\mathbb{R}_{+}^{n}\neq \emptyset $.
\end{thm}

\begin{proof}
We consider the set-valued map $F:M^{n}\times M^{n}\rightrightarrows \mathbb{%
R}$%
\[
F\left( x,y\right) =\left( -\infty ,\sigma \left( C\left( x\right) ,y\right)\right]
\]%
and prove that it satisfies the hypotheses of Theorem \ref{t511}.

In view of Remark \ref{r21} $\forall y\in D,\quad x\longrightarrow F\left(
x,y\right) $ is upper semi-continuous on $M^{n}$ if for any sequence $\left(
x_{n}\right) \in M^{n},x_{n}\longrightarrow x$ and any $b\in R$ such that $%
\sigma \left( C\left( x\right) ,y\right) <b$ we can show that $\sigma \left(
C\left( x_{n}\right) ,y\right) <b$ for $n$ sufficiently large. But this
holds true since the upper hemi-continuity of $C$ together with Remark \ref%
{r22} guarantees that%
\[
\limsup_{x_{n}\longrightarrow x}\sigma \left( C\left( x_{n}\right)
,y\right) \leq \sigma \left( C\left( x\right) ,y\right) .
\]

To see that%
\[
\forall x\in M^{n},\quad y\longrightarrow F\left( x,y\right) \text{ is upper
semi-continuous on }M^{n}\backslash D,
\]%
we rely again on Remark \ref{r21} and also on the fact that $C\left(
x\right) $ is compact and convex and $\sigma \left(
C\left( x\right) ,y\right)$, the support function of  $C(x),$ is continuous. Therefore%
\[
\sigma \left( C\left( x\right) ,y_{n}\right) \longrightarrow \sigma \left(
C\left( x\right) ,y\right) .
\]

For any $x\in D$, the concavity of the set-valued map $y\longrightarrow
F\left( x,y\right) $\ follows from the convexity of single-valued map $%
y\longrightarrow \sigma \left( C\left( x\right) ,y\right) $, which is the pointwise supremum of a family of affine functions.

Finally, the condition that%
\[
\forall x\in D,\quad F\left( x,x\right) \cap \mathbb{R}_{+}\neq \emptyset
\]%
is exactly Walras' law in our hypothesis (ii). So, based on Theorem \ref%
{t511} we conclude that there exists $x_{0}\in M^{n}$ such that%
\[
F\left( x_{0},y\right) \cap \mathbb{R}_{+}^{n}\neq \emptyset \quad \forall
y\in M^{n},
\]%
or in other words%
\begin{equation}
\sigma \left( C\left( x_{0}\right) ,y\right) \geq 0\quad \forall y\in M^{n}.
\label{s>0}
\end{equation}%
But the above inequality is equivalent to%
\begin{equation}
\sigma \left( C\left( x_{0}\right) -\mathbb{R}_{+}^{n},y\right) \geq 0\quad
\forall y\in \mathbb{R}^{n}  \label{sigma}
\end{equation}%
since%
\[
\sigma \left( -\mathbb{R}_{+}^{n},y\right) =\left\{
\begin{array}{cl}
0 & \text{if }y\in \mathbb{R}_{+}^{n} \\
+\infty  & \text{if }y\notin \mathbb{R}_{+}^{n}.%
\end{array}%
\right.
\]%
At this point, we need the fact that $C\left( x_{0}\right) -\mathbb{R}%
_{+}^{n}$ is closed and convex  to conclude from (\ref%
{sigma}) that $0\in C(x_0)-\Real^n_+$ that is
\[
C\left( x_{0}\right) \cap \mathbb{R}_{+}^{n}\neq \emptyset .
\]
\end{proof}

\section{Non-cooperative equilibrium in $n-$person games}

Following the approach of Aubin, we consider a $n$-person game in normal
(strategic) form, (see \cite{Aubin}) and we denote by $E^{i}$ the strategy
set of each player $i$, $i\in \left\{ 1,\ldots ,n\right\} $, while $%
E=\prod\nolimits_{i=1}^{n}E^{i}$ is the set of multistrategies $x=\left(
x^{1},\ldots ,x^{n}\right) $.

In the absence of cooperation, from the perspective of player $i$, the set
of multistrategies can be regarded as a product between the set $E^{i}$ of
strategies that he controls, and the set of strategies $\widehat{x}%
^{i}=\left( x^{1},\ldots ,x^{i-1},x^{i+1},\ldots ,x^{n}\right) $ of all
other players%
\begin{equation*}
E=E^{i}\times \widehat{E}^{i}.
\end{equation*}

The behavior of each player is defined by a loss function $%
f^{i}:E\rightarrow \mathbb{R}$ with associated decision rules%
\begin{equation*}
C^{i}:\widehat{E}^{i}\rightarrow E^{i},C^{i}\left( \widehat{x}^{i}\right)
=\left\{ x^{i}\in E^{i}:f^{i}\left( x^{i},\widehat{x}^{i}\right) =\underset{%
y^{i}\in E^{i}}{\inf }f^{i}\left( y^{i},\widehat{x}^{i}\right) \right\} .
\end{equation*}

A non-cooperative equilibrium (or Nash equilibrium) is a fixed point of the
set-valued map%
\begin{equation*}
C:E\rightrightarrows E,C\left( x\right) =\prod\limits_{i=1}^{n}C^{i}\left(
\widehat{x}^{i}\right) .
\end{equation*}

As shown in \cite{Aubin}, Nash equilibria can be characterized using the map
$\varphi :E\times E\rightarrow \mathbb{R}$ defined by%
\begin{equation*}
\varphi \left( x,y\right) =\sum\limits_{i=1}^{n}\left( f^{i}\left( x^{i},%
\widehat{x}^{i}\right) -f^{i}\left( y^{i},\widehat{x}^{i}\right) \right) .
\end{equation*}

\begin{lem}[\protect\cite{Aubin}]
A multistrategy $x_{0}\in E$ is a non-cooperative equilibrium if and only if%
\begin{equation*}
\varphi \left( x_{0},y\right) \leq 0,\quad \forall y\in E.
\end{equation*}
\end{lem}

Now we can verify the existence of non-cooperative equilibria under
convexity assumptions formulated on self segment-dense subsets of the
strategy sets. This generalizes the classical result of Nash (see \cite%
{Aubin} Theorem 12.2) by allowing that the convexity is violated on small
sets.

\begin{thm}
Suppose that for any $i\in \left\{ 1,\ldots ,n\right\},$ the sets $E^{i}$
are convex and compact and let $D^{i}\subseteq E^{i}$ be self segment-dense
subsets of $E^{i}$. Assume further that for every $i\in \left\{ 1,\ldots
,n\right\}$ the following assumptions hold:

\begin{itemize}
\item[(i)] $f^i$ is lower semicontinuous on $E,$

\item[(ii)] for all $y^i\in D^i$ the map $\widehat{x}^{i}\longrightarrow
f^i(y^i,\widehat{x}^{i})$ is upper semicontinuous on $\widehat{E}^{i},$

\item[(iii)] for all $\widehat{x}^{i}\in \widehat{E}^{i}$ the map $%
y^i\longrightarrow f^i(y^i,\widehat{x}^{i})$ is upper semicontinuous on $%
E^i\setminus D^i,$

\item[iv)] for all $\widehat{x}^{i}\in \widehat{D}^{i} $ the map $%
y^i\longrightarrow f^i(y^i,\widehat{x}^{i})$ is convex on $D^i.$
\end{itemize}

Then there exists a non-cooperative equilibrium.
\end{thm}

\begin{proof}
The theorem follows from Theorem  \ref{c51}.

We consider the set $E$ and the function $\varphi $ defined above. The set $E
$, being a product of compact and convex sets is itself compact and convex, meanwhile the set
$D=\prod\nolimits_{i=1}^{n}D^{i}$ is self segment-dense in $E.$

The assumptions $(i)-(iv)$ assures that the
hypotheses $(i)-(iii)$ of Theorem \ref{c51} are satisfied.  Also, we have that $\varphi \left( x,x\right) =0$ for any
$x\in E$, so the characterization of Nash equilibria together with the
aforementioned abstract result guarantee the existence of a non-cooperative
equilibrium point.
\end{proof}


\section{Conclusions}


In this paper, we provide existence results for equilibrium problems (both single- and setvalued) under convexity and continuity assumptions that do not hold on the whole domain but on a special type of subset that we call self segment-dense. By a counterexample, we show that taking any dense subset is not enough, and that the new concept of a self segment-dense subset is essential in this context. We underline that Theorems \ref{t31} and \ref{t511} extend the original results of A. Krist\'aly and Cs. Varga \cite{KV}, who impose conditions on the whole domain.

By means our aforementioned Ky-Fan-type results we can prove:
a) The existence of an economic equilibrium when the constraint imposed by Walras' law holds just on a self segment-dense subset of the price simplex.
b) The existence of a Nash equilibrium for a non-cooperative $n$-person game under the assumption that the loss function of each player is convex on a self segment-dense subset of the set of strategies, not on the whole set.\\

{\bf Acknowledgements} {This work was supported by a grant of the Romanian Ministry of Education,
CNCS - UEFISCDI, project number PN-II-RU-PD-2012-3 -0166.\\
This research was supported by a grant of the Romanian National Authority
for Scientific Research CNCS - UEFISCDI, project number
PN-II-ID-PCE-2011-3-0094.}

\end{document}